\documentclass{amsart}

\usepackage{amsmath,amstext,amssymb,amsopn,amsthm}
\usepackage{url}
\usepackage{eucal,mathrsfs,dsfont}%gives nicer set-names. Use \mathds instead of \mathds
\usepackage{soul}

\theoremstyle{plain}
\newtheorem{theorem}{Theorem}[section]
\newtheorem{corollary}[theorem]{Corollary}
\newtheorem{lemma}[theorem]{Lemma}

\theoremstyle{remark}

\newtheorem{remark}[theorem]{Remark}
\newtheorem{example}[theorem]{Example}
\newcommand{\eps}{\varepsilon}

\newcommand{\R}{\mathds{R}}
\newcommand{\Rd}{\mathds{R}^d}

\newcommand{\wt}{\widetilde}

\begin{document}

\title{Drift reduction method for SDEs driven by inhomogeneous singular L{\'e}vy noise}

\author[T. Kulczycki]{Tadeusz Kulczycki}
\author[O. Kulyk]{Oleksii Kulyk}
\author[M. Ryznar]{Micha{\l} Ryznar}

\thanks{T. Kulczycki and M. Ryznar were supported in part by the National Science Centre, Poland, grant no. 2019/35/B/ST1/01633}
\address{Faculty of Pure and Applied Mathematics, Wroc{\l}aw University of Science and Technology, Wyb. Wyspia{\'n}skiego 27, 50-370 Wroc{\l}aw, Poland.}
\thanks{O. Kulyk has been supported through
the DFG-NCN Beethoven Classic 3 programme, contract no.
2018/31/G/ST1/02252 (National Science Center, Poland) and SCHI-419/11–1
(DFG, Germany)}
\email{tadeusz.kulczycki@pwr.edu.pl}
\email{oleksii.kulyk@pwr.edu.pl}
\email{michal.ryznar@pwr.edu.pl}

\pagestyle{headings}

%\thankstext{T1}{A sample of additional note to the title.}


\begin{abstract}
We study SDE
$$
d X_t = b(X_t) \, dt + A(X_{t-}) \, d Z_t,
\quad X_{0} = x \in \R^d, \quad t \ge 0
$$
where $Z=(Z^1, \dots, Z^d)^T$, with  $Z^i, i=1,\dots, d$ being independent one-dimensional symmetric jump L\'evy processes, not necessarily identically distributed. In particular, we cover the case when each $Z^i$ is one-dimensional symmetric $\alpha_i$-stable process ($\alpha_i \in (0,2)$ and they are not necessarily equal).

Under certain assumptions on $b$, $A$ and $Z$ we show that the weak solution to the SDE is uniquely defined and Markov, we provide a representation of the transition probability density and we establish H{\"o}lder regularity of the corresponding transition semigroup.

The method we propose is based on a reduction of an SDE with a drift term to another SDE without such a term but with coefficients depending on time variable. Such a method have the same spirit with the classic characteristic method and seems to be of independent interest.
\end{abstract}

\maketitle

%%%%%%%%%%%%%%%%%%%%%%%%%%%%%%%%%%%%%%%%%%%%%%
%%%% Main text entry area:
\section{Introduction}
In this paper we study an SDE of the form
\begin{equation}
\label{main}
d X_t = b(X_t) \, dt + A(X_{t-}) \, d Z_t,
\quad X_{0} = x= (x_1,\dots,x_d)^T \in \R^d, \quad t\ge 0,
\end{equation}
where the driving process $Z$ has the form
\begin{equation}\label{Z}
Z=(Z^1, \dots, Z^d)^T,
\end{equation}
with  $Z^i, i=1,\dots, d$ being independent scalar symmetric L\'evy processes. Such L\'evy noise is \emph{essentially singular} in the terminology of \cite{KKS2021}. Namely, the L\'evy measure $\mu$ of $Z$ is concentrated on the collection of coordinate axes in $\R^d$, which combined with a non-trivial rotation occurring due to the matrix coefficient $A(x)$ yield lack of a single reference measure for the images of $\mu$ under the multiplicative mappings $z\mapsto A(x)z$ for different $x$. Such a structure of the L\'evy noise may result in a substantial novelties in a local behavior of the transition probabilities of the associated processes, e.g. \cite{KRS2021}, \cite{KKS2021}. The situation complicates even more when the components of the noise $Z^1, \dots, Z^d$ are allowed to have different scaling properties, e.g. in the case of $Z^i, i=1, \dots, d$ being scalar (symmetric) $\alpha_i$-stable processes with different indices $\alpha_i$.
For a detailed discussion of the new effects which appear for such \emph{inhomogeneous singular} L\'evy noises we refer for to our recent article \cite{KKR2021}, where SDEs with such noises  were studied.

The SDEs considered in \cite{KKR2021} did not involve a drift term. The method used therein was essentially analytical and based on a version of the \emph{parametrix method} for a construction of the heat kernel for the associated Markov process, treated  as a solution to the Kolmogorov backward differential equation.

In this article we consider a further natural  extension of the model, allowing the SDE \eqref{main} to have a drift term. It is known that, for the analytic methods mentioned above, to include a drift term to an SDE (or, equivalently, a gradient term in the associated PDE) is quite a non-trivial problem in the general setting, where the `order' of the non-local part of the generator is allowed to be smaller than one; that is, when the gradient term is a leading one in the entire generator.   One possibility to treat such models is to adapt the parametrix method for such a setting by introducing a special `flow corrector' method introduced in \cite{KK18}, \cite{K18}, see also a detailed discussion in \cite{KKS2021}. In this paper we develop another possibility,  which we believe to have an independent interest and be potentially useful for other types of SDEs with jumps.  Namely, let $t>0$ be fixed, and we would like to identify the law of the value $X_t$  of the solution to SDE \eqref{main} at this time moment. Denote by $\chi_t(x)$ the flow of solutions to the Cauchy problem
$$
\partial_t \chi_t(x) =  b(\chi_t(x)), \quad t\in \R, \quad \chi_0(x) = x,
$$
and consider the process
$$
X_s^t=\chi_{t-s}(X_s), \quad s\in [0,t].
$$

Clearly, at the terminal point $s=t$ the processes $X^t$ and $X$ are equal, that is $X^t_t=X_t$. On the other hand, under proper assumptions on the coefficients and noise in \eqref{main}, see Lemma \ref{c212} and Remark \ref{rem_c212} below, we can apply the It\^o formula to show that $X^t$ satisfies the SDE
\begin{equation}\label{main1}
 d X_s^t=\int_{\R^d}V^t_s(X_{s-}^t,z)N(ds, dz), \quad s\in [0,t], \quad X_0^t=\chi_t(x),
\end{equation}
where $N(ds,dz)$ is the Poisson random measure which corresponds to the L\'evy process $Z$ from \eqref{main} in the usual sense that
 $$
 dZ_s=\int_{\R^d}z N(ds, dz),
 $$
 and the coefficient $V^t_s(x,z)$ has the form
$$
V^t_s(x,z)=\chi_{t-s}(\kappa_{t-s}(x)+A(\kappa_{t-s}(x))z)-x,
$$
here and below we denote $\kappa_t(x)=\chi_{-t}(x)$, the \emph{inverse flow} for $\chi_t(x)$.

Next, the process $X^*_s=X^t_{t+s}, s\in [-t,0]$ satisfies the SDE
\begin{equation}\label{main2}
 d X_s^*=\int_{\R^d}V_s(X_{s-}^*,z)N^*(ds, dz), \quad s\in [-t,0], \quad X_{-t}^*=\chi_t(x)
\end{equation}
with
$$
V_s(x,z)=\kappa_{s}(\chi_{s}(x)+A(\chi_{s}(x))z)-x
$$
and a new Poisson point measure $N^*(ds, dz)$ which has the same intensity measure $ds\mu(dz)$ with the original $N(ds, dz).$   SDE \eqref{main2} does not have a drift term, and by the results of \cite{KKR2021} its solution is weakly unique and is a (time-inhomogeneous) Markov process with a transition probability  density $p_{t,s}(x,y)$.  Then the above stochastic equivalences immediately yield the same set of results for the initial equation  \eqref{main}. Namely, the solution to \eqref{main} is weakly unique, and its transition probability density can be expressed as
\begin{equation}\label{heat}
u_{t}(x,y) = p_{ -t, 0}(\chi_{t}(x),y).
\end{equation}

To summarize, the method we propose is based on a reduction,  by
a proper change of variables, of an SDE with a drift term to another SDE free from such a term.  Such a reduction seem to have the same spirit with the classic \emph{characteristic method}, which reduces a 1st order quasi-linear PDE to a family of ODEs on \emph{characteristic curves}. This  similarity becomes even better visible if we consider the Kolmogorov differential equation associated with \eqref{main}, see Section \ref{s24} below. We believe that such \emph{stochastic characteristic method} has an independent interest and is potentially  useful for  other types of SDEs with jumps.

The study of SDEs driven by singular L{\'e}vy noises have attracted attention in recent years. In \cite{BC2006} the weak
well-posedness for SDE (\ref{main}) was established under the assumption that $A(x)$ is bounded and non-degenerate for each $x \in \R^d$, $x \to A(x)$ is continuous, $b \equiv 0$ and $Z$ is a cylindrical $\alpha$-stable process. Strong Feller property for SDEs driven by additive cylindrical L{\'e}vy processes have been studied in \cite{PZ2011}. The existence of densities for SDEs driven by singular L\'evy processes have been studied in \cite{DF13} and \cite{FJR2021}. The strong well-posedness for SDEs driven by $\alpha$-stable-like L{\'e}vy processes (including cylindrical case) with H{\"o}lder drifts was established in \cite{CZZ2021}, giving an affirmative answer to an open problem proposed by Priola. In \cite{KKS2021} different stable-like models have been treated, where the stability index and the spherical kernel (i.e. the distribution of the jump direction) are space-dependent. Existence and representation of transition density and strong Feller property of the related semigroup have been established. The main idea of that paper, which can be heuristically described as a dynamical truncation of L{\'e}vy measures, is crucial for the results obtained in \cite{KKR2021} and indirectly for the results in our recent paper.
Properties of transition densities and H{\"o}lder regularity of transition semigroups for SDEs driven by singular L{\'e}vy noices under various assumptions have been studied in \cite{KRS2021}, \cite{KR2020}, \cite{CHZ2020} and \cite{KKR2021}.

The rest of the paper is organized as follows. In Section \ref{s2} we introduce notation, present assumptions and formulate main results. Section \ref{sec_Proofs} contains proofs, its most important part is Section \ref{s24}, where we explain the method of solution. In Section \ref{sec_Examples} we give examples of processes $Z$, drifts $b$ and matrices $A$ satisfying our assumptions.

\section{Main results}\label{s2}
\subsection{Notation and assumptions} In this section, we collect all the assumptions we impose on our model.

Notation:

\begin{itemize}
	\item All vectors $x\in \R^d$ are column type. For any $x,y \in \R^d$ by $x\cdot y$ we denote the standard scalar product of $x$ and $y$.
	\item For a function $f: \R^d\mapsto \R$ its gradient is defined $\nabla f(x)=\begin{pmatrix} \frac{\partial f}{\partial x_1} \\ \dots\\  \frac{\partial f}{\partial x_d} \end{pmatrix}$
	\item For a real $d \times d$ matrix $ A$ by $ |A|$ we denote its operator norm that is $ |A|=\sup_{|x|\le 1}|Ax|$
	\item For a function $h=\begin{pmatrix}  h_{1}\\ \dots\\  h_{d} \end{pmatrix}: \R^d\mapsto \R^d$, by
	$$Dh=\left[ \nabla h_{1}, \dots,  \nabla h_{d}\right]^T$$ we denote its derivative, that is  a $d\times d$ matrix which rows are equal to  $[\nabla h_k]^T, 1\le k\le d$. Note that, if $g:  \R^d\mapsto \R^d$, then
	we have the following formula for the derivative of the  composition $g$ with $h$, $w(x)= g(h(x))$,
	$$D w(x) =  Dg( h(x)) D h(x).$$
	\item  Note that, if $f:  \R^d\mapsto \R$, then
	we have the following formula for the gradient  of the  composition $f$ with $h$, $w(x)= f(h(x))$,
	$$\nabla w(x) =  [D h(x)]^T \nabla f( h(x)).$$
	\item In the whole paper we assume that constants denoted by $c, c_1, c_2, \ldots$ are positive and we adopt the convention that these constants may change their value from one use to the next.
	\item We denote $a \wedge b = \min(a,b)$, $a \vee b = \max(a,b)$.
\end{itemize}

Now we start the description of scalar L\'evy processes involved, as the coordinates, in the representation \eqref{Z}.
Let the characteristic exponent $\psi$ of a one-dimensional, symmetric L\'evy  process be given by
$\displaystyle{\psi(\xi) = \int_{\R} (1 - \cos(\xi x)) \nu(dx) }$,
where $\nu$ is a  symmetric, infinite L{\'e}vy measure. The corresponding  \emph{Pruitt function}  $h(r)$ is given by
$$
h(r) = \int_{\R} (1 \wedge (|x|^2 r^{-2})) \nu(d x), \quad r > 0.
$$

We will assume the following scaling conditions for the function $h$: for some $0<\alpha\leq \beta\leq 2$ and $0<C_1\le 1\le C_2<\infty$,
\begin{equation}
\label{h-scaling_u0}
C_1 {\lambda}^{-\alpha} h\left({r}\right) \le h\left(\lambda r\right) \le
C_2 {\lambda}^{-\beta} h\left({r}\right), \  0<r \le 1,\ 0<\lambda \le 1.
\end{equation}

It is well known (see e.g. \cite[Section 4]{KKR2021}) that the above assumption is equivalent to the following weak scaling property for $\psi$: there are constants $0<C_1^*\le 1\le C_2^*<\infty$,

\begin{equation}
\label{psi-scaling}
C_1^* {\lambda}^{\alpha} \psi\left({\xi}\right) \le \psi\left(\lambda \xi\right) \le
C_2^* {\lambda}^{\beta}  \psi\left({\xi}\right), \  |\xi| \ge 1,\ \lambda \ge 1.
\end{equation}

Once the condition  \eqref{h-scaling_u0} (or equivalently \eqref{psi-scaling})  is  satisfied, we say that the characteristic exponent $\psi$ (or the L\'evy measure $\nu$) have the weak scaling property with indices $\alpha, \beta$, and write $\psi\in \mathrm{WSC}(\alpha, \beta)$ (resp. $\nu\in \mathrm{WSC}(\alpha, \beta)$).
We also  observe that, by \cite[Corollary 6]{BGR14}, the L\'evy measure $\nu$ satisfies
\begin{equation}
\label{bound_nu}
\int_{|z|\leq 1}|z|^{\gamma}\nu(dz)<\infty,
\end{equation}
for any $\gamma>\beta$.

By $\psi_i$, $\nu_i$ and $h_i$ we denote corresponding characteristic exponents, L{\'e}vy measures and Pruitt functions of coordinates $Z^i$ of the process $Z = (Z^1, \ldots,Z^d)$ ; recall that these coordinates are assumed to be independent.

We will consider two cases:
\vskip 10pt\noindent
\textbf{(A)} All characteristic exponents $\psi_i, i=1,\dots, d$ are equal and $\psi_1 \in \mathrm{WSC}(\alpha, \beta)$.

\vskip 10pt\noindent\textbf{(B)}
Characteristic exponents $\psi_i, i=1,\dots, d$ are not  all  the same , i.e. there exist $i,j$ such that $\psi_i\not=\psi_j$,    and $\psi_i\in \mathrm{WSC}(\alpha, \beta)$, $i, j \in \{1,\ldots,d\}$.
\vskip 10pt

In both of these cases, the process $Z$ has the transition density $\wt {G}_t(x,y)=\wt {G}_t(y-x),$
where
$$
\wt {G}_t(w)=\prod_{i=1}^{d}\wt {g}_t^i(w_i), \quad w=(w_1,\dots, w_d)^T\in \R^d,
$$
and $\wt{g}_t^i, i=1,\dots,d$ are the distribution densities for the coordinates (all $\wt{g}^i_t$ are the same in the case \textbf{(A)}).

Next, we assume the following conditions on the coefficients.

\vskip 10pt\noindent \textbf{(C)}
For any $x \in \R^d$ $A(x) = (a_{i,j}(x))$ is a $d \times d$ matrix and there are constants $C_3, C_4, C_5 > 0$, $\eta_1 \in (0,1]$ such that for any $t \geq  0$, $x, y \in \R^d$, $i, j \in \{1,\ldots,d\}$,
\begin{equation}
\label{bounded}
|A(x)| \le C_3,
\end{equation}
\begin{equation}
\label{determinant}
|\det(A(x))| \ge C_4,
\end{equation}
\begin{equation}
\label{Holder}
|A(x)-A(y)|\le C_5 |x - y|^{\eta_1},
\end{equation}
The function $\R^d \ni x \to b(x) \in \R^d$ is differentiable and there are constants $C_6, C_7 > 0$ and
\begin{equation}
\label{eta2}
\eta_2 >\max(0, (\beta-1))
\end{equation}
such that for any  $x, y \in \R^d$,
\begin{equation}
\label{b_grad_est}
|Db(x) | \le C_6,
\end{equation}
\begin{equation}
\label{b_estimate}
|Db(x) -Db(y)| \le C_7 |x-y|^{\eta_2}.
\end{equation}

\vskip 10pt
Now we introduce functions $\kappa_t(x)$, $\chi_t(x)$ and a family of matrices $A_t(x)$ which play a crucial role in our paper.
Let $\kappa$, $\chi$ be solutions of the following differential equations
$$
\partial_t \kappa_t(x) = - b(\kappa_t(x)), \quad \text{for $t \in \R $, $x \in \R^d$},
$$
$$
\partial_t \chi_t(x) =  b(\chi_t(x)), \quad \text{for $t \in \R $, $x \in \R^d$},
$$
with initial conditions $\kappa_0(x) = \chi_0(x) = x$, for $x \in \R^d$.
Note that $\chi_t(x) = \kappa_{-t}(x)$ for $t\in \R$, $x \in \R^d$. It is a standard fact that if assumptions \textbf{(C)} are satified then the above initial value problems have unique solutions. Properties of $\kappa_t(x)$, $\chi_t(x)$ are discussed in Section \ref{flow}.
For $t \in \R $, $x, y \in \R^d$ we define
$$
 A_t(x) = (D \kappa_{t})(\chi_{t}(x)) A(\chi_{t}(x)).
$$

We need to impose a further assumption on the coefficients, which we present in the following two slightly  different forms.
{ \vskip 10pt\noindent \textbf{(D)}
There is a  constant $C_{8} > 0$ such that for any  $x \in \R^d$
\begin{equation}
\label{b_est_max}
|b(x)| \le C_{8}.
\end{equation}

\textbf{(D)} combined with \textbf{(C)} yields the following:
{ \vskip 10pt\noindent \textbf{(D')}}
There is constant $C_9 > 0$ such that for any  $t, s \in \R$ and $x\in \R^d$
\begin{equation}
\label{D11}
|A_t(x) - A_s(x)| \le C_9 e^{C_9(|t| \vee |s|)} |s - t|^{\eta_1 \wedge \eta_2}.
\end{equation}

The justification that \textbf{(D)} combined with \textbf{(C)} implies \textbf{(D')} is contained in Section \ref{sec_Proofs}, see Lemma \ref{CD}. Example \ref{Ex_D} shows that \textbf{(D')} is strictly weaker than \textbf{(D)}.

\vskip 10pt
In the case \textbf{(A)}, the H\"older index $\eta_1$ can be arbitrarily small. In the case \textbf{(B)}, $\alpha$, $\beta$, $\eta_1$ and $\eta_2$ should satisfy certain additional assumptions. Namely, we assume the following

\vskip 10pt\noindent\textbf{(E)}
\begin{equation}
\label{indices}
\frac{\beta}{\alpha} < 1 + (\eta_1 \wedge \eta_2), \quad \frac{1}{\alpha} - \frac{1}{\beta} < \eta_1 \wedge \eta_2.
\end{equation}

%{\tk{ \begin{remark}
%\label{remarkD1D2}
%Assumption \textbf{(D2)} is stronger than \textbf{(D1)} in the following sense. One can show that if \textbf{(C)} and \textbf{(D2)} hold then \textbf{(D1)} is satisfied (with constant $C_8$ depending only on $d$, $\eta_1$, $\eta_2$, $C_1, \ldots, C_7$, $C_{9}$).
%\end{remark}
%This remark will be justified in Preliminaries.
%}}

For abbreviation, for any $u > 0$ we will use the notation
$$
\sigma(u) = (u, h_1(1), \ldots, h_d(1),  h^{-1}_1(1), \ldots, h^{-1}_d(1), h_1^{-1}(1/u), \ldots, h_d^{-1}(1/u)).
$$

\subsection{Main statements} In this section, we formulate the main statements of the paper.

For $t > 0$, $x, y \in \R^d$ define

\begin{equation}
\label{Euler}
\wt{u}_t(x,y) = \frac{1}{|\det A(\chi_t(x))|}\wt {G}_{t}(A^{-1}(\chi_t(x))(y-\chi_t(x))),
\end{equation}
where $\wt {G}_t(\cdot)$ is the distribution density of $Z_t$.

%\mr{Note that in the case $b(x)=x$   we have $det (D \kappa_{t})(\chi_{t}(x))= e^t\to 0, \quad t\to -\infty$.
%This implies that $det A_t(x)\to 0, \quad t\to -\infty$ and the assumption (9) from \cite{KKR2021} is not satisfied. However I do not think we should worry about this, since the assumption (9) from \cite{KKR2021} is overstated - probably it is enough that it holds locally in $t$, but it should be checked. }

\begin{theorem}
\label{main_thm} Assume either \textbf{(A)}, \textbf{(C)}, \textbf{(D')} or \textbf{(B)}, \textbf{(C)}, \textbf{(D')}, \textbf{(E)}. Then for any $ x\in \R^d$ the SDE (\ref{main}) has a unique weak solution $X$. The process $X$ is a Markov process which has a transition density $u_t(x,y)$. The transition density admits a representation
\begin{equation}\label{representation}
u_t(x,y)=\wt u_{t}(x,y)+\wt q_{t}(x,y),\quad x,y\in \R^d, \quad t > 0 ,
\end{equation}
where $\wt u_{t}(x,y)$ is given by \eqref{Euler} and the residual part $\wt q_{t}(x,y)$ satisfies
\begin{equation}
\label{q_estimate}
\int_{\R^d} |\wt q_{t}(x,y)| \, dy \leq c \, t^{\eps_0}, \quad x\in \R^d,
\end{equation}
where $\eps_0$ is defined in Remark \ref{remark_epsilon} and
the constant $c$ depends only on $d$, $\alpha$, $\beta$, $\eta_1$, $\eta_2$, $C_1, \ldots, C_7$, $C_9$, $h_1(1), \dots, h_d(1)$.
\end{theorem}
\begin{remark}
\label{rem_main_thm}
One may think of $\wt u_{t}(x,y)$ as the "principal part" of the transition density and $\wt q_{t}(x,y)$ as the corresponding "residual part". If additionally \textbf{(D)} is satisfied, then (\ref{representation}) and (\ref{q_estimate})  remain true with the principal part in the following simpler form
\begin{equation*}
\check{u}_t(x,y) = \frac{1}{|\det A(x)|}\wt {G}_{t}(A^{-1}(x)(y-\chi_t(x))).
\end{equation*}
\end{remark}

Denote by $\mu$ the L\'evy measure of the process $Z$, and define for $t \in \R$, $x, z \in \R^d$
$$
T^{t,z}f(x)= f(x+V_t(x,z)),
$$
where we recall that
$$
 V_t(x,z) = \kappa_{t}(\chi_{t}(x) + A(\chi_{t}(x)) z) - x.
$$

Assume the following.
\vskip 10pt\noindent
{\bf{(I)}} For all $t$ and $\mu$-a.a. $z$,  $T^{t,z}$ is a bounded linear operator in $L_{1}(\R^d)$, and there exists $C_0<\infty$ such that
$$
\|T^{t,z}\|_{L_1\to L_1}\leq C_0, \quad t\in \R, \quad z\in \mathrm{supp}\, \mu.
$$

We have the following representation of the transition density.
\begin{theorem}
\label{main_thm2}
Let the conditions of Theorem \ref{main_thm} and additional assumption \textbf{(I)} hold. Then the density
$u_t(x,y)$ is bounded, that is
$$\sup_{x,y\in \R^d} u_{t}(x,y)< \infty, \quad 0 < t < \infty.$$
 Moreover, for any $\tau > 0$   there exists  $c>0$, depending only on
 $d$, $\alpha$, $\beta$, $\eta_1$, $\eta_2$, $C_0, \ldots, C_7$, $C_9$, $\sigma(\tau)$
such that  the residual term in the representation \eqref{representation} satisfies

\begin{equation}
\label{q_pointwise_estimate}
|\wt q_{t}(x,y)|\leq c\wt {G}_{t}(0) t^{\eps_0}, \quad 0< t \le \tau,\quad  x,y\in \R^d.
\end{equation}

\end{theorem}

Define  by $\{U_{t}\}$ the \emph{evolutionary family} corresponding to the process $X$ in the usual way: for any $0 < t < T$, $x\in \R^d$ and a bounded Borel function $f: \R^d \to \R$,
$$
U_{t} f(x) = \int_{\R^d} u_{t}(x,y) f(y) \, dy.
$$
Under just the basic conditions of Theorem \ref{main_thm}, we prove  H\"older continuity of this evolutionary family.
\begin{theorem}
\label{Holder_thm_new}
Assume either \textbf{(A)},\textbf{(C)}, \textbf{(D')} or \textbf{(B)}, \textbf{(C)}, \textbf{(D')}, \textbf{(E)}.
For any $0 < \gamma < \gamma' < \alpha$, $\gamma \le 1$,  $0<t < \tau$,  $x, y \in \Rd$ and a bounded Borel function $f: \R^d \to \R$ we have
$$
\left|U_{t} f(x) - U_{t}f(y)\right|
\le c |x - y|^{\gamma} t^{-\gamma'/\alpha} \|f\|_{\infty},
$$
where $c$ depends only on $\gamma$, $\gamma'$, $d$, $\alpha$, $\beta$, $\eta_1$, $\eta_2$, $C_1, \ldots, C_7$, $C_9$, $\sigma(\tau)$.
\end{theorem}

\begin{remark}
\label{remark_epsilon}
The constant $\eps_0$, which appears in above theorems is chosen in the following way. Put $\eta = \eta_1 \wedge \eta_2$.
 In the case (A) we set
\begin{eqnarray*}\eps_0&=& \min \left\{ \frac{\eta \alpha}{2(d+3)\beta}, \frac{\eta \alpha}{2(d+3)},
\frac{  \eta/(\beta(1+\eta))}{2+2/\alpha +  \eta/(\beta(1+\eta))  },\frac {\eta_2}{\eta_2+\beta(1+(d+1)/\alpha)}\right\}, \end{eqnarray*}
while in the case (B) we pick
\begin{eqnarray*}\eps_0&=& \min \left\{ \frac{(1 + \eta)/\beta - 1/\alpha}{2(d+3)/\alpha}, \frac{\eta - \left(1/\alpha - 1/\beta\right)}{2(d+3)/\alpha},
\frac{ 1/\beta - 1/((1+\eta)\alpha)}{2+2/\alpha +1/\beta - 1/((1+\eta)\alpha)},\right.\\
&& \quad \quad \quad \left.
\frac {1+\eta_2-\beta/\alpha}{1+\eta_2+\beta(1+d/\alpha)}\right\}. \end{eqnarray*}
Due to our assumptions $\eps_0$ is positive.

Such choice of $\eps_0$ follows from \cite[Remark 5.1]{KKR2021} and Lemma \ref{assumptions}.
\end{remark}

\section{Proofs}
\label{sec_Proofs}

\subsection{Classes of functions} In this section, we define some classes of real functions which will be needed in the sequel.

For $f \in C^1(\R^d)$ and $\eta, r >0$ we put
$$
\rho_{\eta,r}(f)=\sup_{x,y:\ 0<|x- y|< r} \frac{|\nabla f(x)-\nabla f(y)|}{|x-y|^{\eta}}.
$$

By $C_b^1(\R^d)$ we understand the class of functions which are in $C^1(\R^d) \cap C_b(\R^d)$ and have all first order derivatives bounded.
For $\eta > 0$ we define
%$$
%C^{1,\eta}(\R^d) = \{f \in C^1(\R^d): \, \|\rho_{\eta,1}(f)\|_{\infty} < \infty\}
%$$
$$
C_b^{1,\eta}(\R^d) = \{f \in C_b^1(\R^d): \, \rho_{\eta,1}(f) < \infty\},
$$
%$$
%C_{\infty}^{1,\eta}(\R^d) = \{f \in C_\infty^1(\R^d): \, \|\rho_{\eta,1}(f)\|_{\infty} < \infty, \, \lim_{|x| \to \infty} \rho_{\eta,1}(f)(x) = 0\}.
%$$

Next, we define classes of functions of two variables $t$ and $x$. Let $\eta, r > 0$ and $\text{I} \subset \R$ be an interval. For $f: \text{I} \times \R^d \to \R$ such that $f(t,x)$ is $C^1$ with respect to $t \in \text{I}$ we put
$$
D_1 f(t,x) = \partial_t f(t,x).
$$
For $f: \text{I} \times \R^d \to \R$ such that $f(t,x)$ is $C^1$ with respect to $x \in \R^d$ we put
$$
D_2 f(t,x) = \nabla_x f(t,x)
$$
and
$$
\rho^I_{\eta,r}(f)=\sup_{t,x,y:0<|x- y|< r, t\in I} \frac{|D_2 f(t,x)-D_2 f(t,y)|}{|x-y|^{\eta}}.
$$

$C_b^{1,(1,\eta_2)}(\text{I} \times \R^d)$ is the class of functions $f(t,x)$, $f: \text{I} \times \R^d \to \R$, such that $f$ is $C^1$ in $t$, $C^1$ in $x$, has its first order derivatives continuous on $\text{I} \times \R^d$, for any fixed $t \in \text{I}$ the function and the derivatives are in $C_b(\R^d)$ and $ \rho^I_{\eta,1}(f) < \infty$.

%$C_{\infty}^{1,(1,\eta)}(\text{I} \times \R^d)$ is the class of functions $f(t,x)$, $f: \text{I} \times \R^d \to \R$, such that $f$ is $C^1$ in $t$, $C^1$ in $x$, has its first order derivatives continuous on $\text{I} \times \R^d$, for any fixed $t \in \text{I}$ the function and the derivatives are in $C_{\infty}(\R^d)$,  $\forall t \in \text{I}$ $\sup_{x \in \R^d} \rho_{\eta,1}(f)(t,x) < \infty$ and $\forall t \in \text{I}$ $\lim_{|x| \to \infty} \rho_{\eta,1}(f)(t,x) = 0$.

%$C_{\infty}^{1,2}(\text{I} \times \R^d)$ is the class of functions $f(t,x)$, $f: \text{I} \times \R^d \to \R$, such that $f$ is $C^1$ in $t$, $C^2$ in $x$, has its first order derivatives in $(t,x)$ and second derivatives in $x$ continuous on $\text{I} \times \R^d$ and for any fixed $t \in \text{I}$ the function and the derivatives are in $C_{\infty}(\R^d)$.

\begin{remark}
\label{r_dependence}
It is easy to note that in the above definitions the function $\rho_{\eta,1}$ can be replaced by $\rho_{\eta,r}$ for some $r>0$.
\end{remark}
\begin{remark}\label{Eta_bound} Suppose that $f\in C^{1,\eta}(\R^d)$. If $|x-y|\le r$,  then
\begin{equation}
\label{eta_bound}|f(y)-f(x)- \nabla f(x)\cdot (y-x)|\le \rho_{\eta,r}(f) |y-x|^{1+\eta}.
\end{equation}
\end{remark}
\begin{proof} By Taylor's formula with the remainder of order $1$ we have
$$f(y)-f(x)= \int_0^1 \nabla f(x+t(y-x))\cdot (y-x)dt.$$ Hence
$$f(y)-f(x)- \nabla f(x)\cdot (y-x)= \int_0^1 [\nabla f(x+t(y-x))-\nabla f(x)]\cdot (y-x)dt.$$
This yields
\begin{eqnarray*}\left|f(y)-f(x)- \nabla f(x)\cdot (y-x)\right|&=& \left|\int_0^1 [\nabla f(x+t(y-x))-\nabla f(x)]\cdot (y-x)dt\right|\\
&\le& |y-x|^{1+\eta}\int_0^1 \frac{\left|\nabla f(x+t(y-x))-\nabla f(x)\right|}{|y-x|^{\eta}}dt,\end{eqnarray*}
which completes the proof.
\end{proof}

\subsection{Main idea of the proof}\label{s24}
In this section we will explain the method of the solution. In principal, the idea is to transform our SDE (\ref{main}) (with a drift) to the auxiliary time-inhomogeneous SDE (without a drift).% \ok{TWO SENTENCES REMOVED}

%The existence and properties of the solution of this last equation is well understood by the results of \cite{KKR2021}. It is worth to add that %in \cite{KKR2021} were used deep ideas from \cite{KKS2021}}.

Let us define the operator $Q$
$$
Q f(x) = \nabla f(x)  \cdot b(x) + Q^{(jump)} f(x),
$$
$$\begin{aligned}
Q^{(jump)} f(x) &=\mathrm{P.V.}\int_{\R^d}\Big(f(x+A(x) z)-f(x)\Big)\mu(d z)
\\&=\sum_{k=1}^d\int_{\R}\Big(f(x+v A(x)\mathbf{e}_k)-f(x)
\\&\hspace*{3.5cm}-v 1_{|v|\leq 1}\nabla f(x)\cdot A(x)\mathbf{e}_k\Big)\nu_k(d v),
\end{aligned}
$$
where $\mu$ is the L\'evy measure of the process $Z$, $\nu_k$ is the L\'evy measure of the $k$-th component $Z^k, k=1, \dots, d$, and $\mathrm{P.V.}$ means that the first integral is taken in the principal value sense. One can easily show, using (\ref{bound_nu}) and (\ref{eta_bound}),  that for each $f \in C_ {b}^{1,\eta_2}(\R^d)$ we have $Q f \in C_b(\R^d)$; recall that $1+\eta_2>\beta$.

One can naturally expect that, once the solution $X$ to \eqref{main} is well defined for $0 \le t $ and is a Markov process, the operator $Q$ should be its generator. Corresponding \emph{Kolmogorov's backward differential equation} for the transition probability density $u_t(x,y)$ of $X$ has the form
\begin{equation}\label{forward}
\partial_t u_{t}(x,y)= Q_x u_t(x,y), \quad  t\ge 0 , \quad x,y \in \R^d,
\end{equation}
here and below $x$ at the operator $Q_{x}$ indicates that the operator $Q$ is applied with respect to the variable $x$. Together with the initial condition
\begin{equation}\label{initial}
u_{t}(x,y) \to \delta_x(y), \quad t\to 0^+,
\end{equation}
this actually gives a hope that  $u_t(x,y)$ can be identified as the \emph{fundamental solution} to the parabolic Cauchy problem for the operator $Q$.

As we have outlined in the Introduction, we would like to avoid solving the above Cauchy problem directly and to use already available results from \cite{KKR2021} instead. For that,  we consider the  (time-dependent) operator family  $\{L_t\}$
\begin{equation}\label{Lt}
\begin{aligned}
L_tf(x)&=\mathrm{P.V.}\int_{\R^d}\Big(f(x+V_t(x,z))-f(x)\Big)\mu(d z), \quad t\in \R,\quad x \in \R^d,
\end{aligned}
\end{equation}
recall that $V_t(x,z) = \kappa_{t}(\chi_{t}(x) + A(\chi_{t}(x)) z) - x$ for $t \in \R$, $x, z \in \R^d$.
One can easily show, using (\ref{bound_nu}) and (\ref{eta_bound}), that $L_t f\in C_b(\R^d)$.
It is clear (e.g. by the virtue of the It\^o formula) that the operator $L_t$ is a generator for the (time-inhomogeneous) Markov process defined by \eqref{main2}.

The following lemma relates  the new operator family $\{L_t\}$ with the original generator $Q$.

%\mr {We have a basic lemma:}
\begin{lemma}\label{l210}
\label{basic}
Let  $f \in C_b^{1,(1,\eta_2)}(\text{I} \times \R^d)$, where $\text{I} \subset \R$ is an interval. Put $g(t,x) = f(t,\kappa_t(x))$. Then for any $t \in \text{I}$ and $x \in \R^d$ we have
$$
\partial_t g(t,x)+ Q_xg(t,x)= (D_1 f)(t,\kappa_t(x))+ (L_{t,x} f)(t,\kappa_t(x)).
$$
\end{lemma}

\begin{proof}
We have
\begin{eqnarray}
\nonumber
&& (L_{t,x}f)(t,\kappa_t(x))\\
\nonumber
&& = \mathrm{P.V.}\int_{\R^d}\Big(f(t,\kappa_t(x)+V_t(\kappa_t(x),z))-f(t,\kappa_t(x))\Big)\mu(d z)\\
\nonumber
&& = \mathrm{P.V.}\int_{\R^d} \left(f\left(t,\kappa_t(x) + \kappa_t(\chi_t(\kappa_t(x)) +A(\chi_t(\kappa_t(x)))z -\kappa_t(x))\right) - f(t,\kappa_t(x))\right) \mu(d z)\\
\nonumber
&& = \mathrm{P.V.}\int_{\R^d} \left(f(t,\kappa_t(x+A(x)z)) - f(t,\kappa_t(x))\right) \mu(d z)\\
\nonumber
&& = \mathrm{P.V.}\int_{\R^d} \left(g(t,x+A(x)z) - g(t,x)\right) \mu(d z)\\
\label{rep2}
&& = Q_x^{(jump)}g(t,x).
\end{eqnarray}

By Lemma \ref{chi_derivative} we get
\begin{eqnarray*}
Q_xg(t,x)&=& Q_x^{(jump)}g(t,x) + [D\kappa_t(x)]^T D_2 f(t,\kappa_t(x))\cdot b(x)\\
&=& Q_x^{(jump)}g(t,x) + D_2 f(t,\kappa_t(x))\cdot D\kappa_t(x) b(x) \\
&=& Q_x^{(jump)}g(t,x) - D_2 f(t,\kappa_t(x))\cdot \partial_t \kappa_t(x).
\end{eqnarray*}
Also note that
$$
\partial_t g(t,x)= D_1 f(t,\kappa_t(x))+ D_2 f(t,\kappa_t(x))\cdot \partial_t \kappa_t(x).
$$
Applying (\ref{rep2}) we end the proof.
\end{proof}

\begin{remark} We can interpret the statement of the lemma as  follows: by changing the variables
$$
(t,x)\rightsquigarrow (t^*, x^*)=(t, \kappa_t(x)),
$$
we transform the operator $\partial_t+Q$  into the operator $\partial_t+L_t$, where the (time-dependent)  PDO $L_t$, on the contrary to the original PDO $Q$, does not contain a gradient term. This change of variables well corresponds to the classical characteristic method for 1st order PDEs.
\end{remark}

Using Lemma \ref{l210}, one can try  construct a fundamental solution to the parabolic Cauchy problem for the operator $Q$ once such a solution for $L_t$ is available. Namely, the function $u_{t}(x,y)$ given by \eqref{heat} is a natural candidate for solving \eqref{forward}, \eqref{initial} since $p_{t,s}(x,y)$ is the heat kernel corresponding to $\partial_t+L_t$. However, there are two substantial difficulties for using such analytic approach directly in order to identify $u_{t}(x,y)$ as the transition distribution density for \eqref{main}. First, the transition density  $p_{t,s}(x,y)$ constructed in \cite{KKR2021}  is proved to satisfy  Kolmogorov's backward differential equation  only in a certain approximate sense, see \cite[Lemma 3.3]{KKR2021}, and may fail have enough regularity in $x$ for Lemma \ref{l210} to be applicable.  Second, even if one manages to show that the formula \eqref{heat} gives \emph{some} fundamental solution, this is not yet enough to identify the law of the solution to \eqref{main} until  uniqueness of such a solution is proved. We avoid these substantial analytical complications by using stochastic calculus arguments outlined in the Introduction. Namely, under the assumptions of Theorem \ref{main_thm}, we have the following

\begin{lemma}\label{c212} Let $X$ be a weak solution to SDE \eqref{main} and $t>0$ be fixed. Then the It\^o formula is applicable to this process and the function $f(s,x)=\chi_{t-s}(x)$, implying that the process $\chi_{t-s}(X_s), s\in [0,t]$ satisfies SDE \eqref{main1}.
\end{lemma}
\begin{remark}
\label{rem_c212}  SDE \eqref{main1} requires a certain care to be properly interpreted, since the term $V^t_s(X_{s-}^t,z)$ may fail to be absolutely integrable on the set $\{|z|\leq 1\}$. We recall that $\mu(dz)$ is symmetric, hence
$$
\mathrm{P.V.}\int_{|z|\leq 1}A_{t-s}(X_{s-}^t) z\mu(dz)=0.
$$
Using this, we write \eqref{main1} in the following more cumbersome but completely rigorous form, which will be actually proved below:
\begin{equation}\label{main1_bis}
\begin{aligned}
 d X_s^t&=\int_{|z|\leq 1} V^t_s(X_{s-}^t,z)\,\widetilde{N}(ds, dz)
 \\&+\int_{|z|\leq 1} \Big(V^t_s(X_{s-}^t,z)-A_{t-s}(X_{s-}^t) z\Big) \, ds\mu(dz)
 \\&+\int_{|z|> 1} V^t_s(X_{s-}^t,z)\,N(ds, dz), \hspace*{2cm} s\in [0,t], \quad X_0^t=\chi_t(x),
 \end{aligned}
\end{equation}
where  $\widetilde{N}$ is the compensated random measure corresponding to $N$.
\end{remark}
\begin{proof}[Proof of Lemma \ref{c212}] First, we recall that by the standard It\^o formula for any $\R^d$-valued function $f(s,x)=\left(f^1(s,x),\dots,f^d(s,x)\right)$ such that all its coordinates are from the class $C^{1,2}([0,t]\times \R^d)$, bounded with their derivatives, the process
$$
Y_s=f(s, X_s), \quad s\in [0,t]
$$
satisfies
\begin{equation}\label{Ito}\begin{aligned}
  Y_s&-Y_0=\int_0^s\Big[ (D_1f)(r, X_r)+  {(D_2f)(r, X_r)b(X_r)}\Big]\, dr
  \\&+\int_0^s\int_{|z|\leq 1}\Big[f(r, X_{r-}+A( X_{r-})z)-f(r, X_{r-})\Big]\widetilde{N}(dr,dz)
  \\&+
  \int_0^s\int_{|z|>1}\Big[f(r, X_{r-}+A( X_{r-})z)-f(r, X_{r-})\Big]{N}(dr,dz)
  \\&+ \int_0^s\int_{|z|\leq 1}\Big[f(r, X_{r}+A( X_{r})z)-f(r, X_{r})-{(D_2f)(r, X_r)A(X_r)z}\Big]\,\mu(dz)\, dr,
 \end{aligned}
\end{equation}
where  $(D_2f)(r, x)= [(D_2f^1)(r, x), \dots, (D_2f^d)(r, x)]^T$ for  $r\in \R, x\in\R^d$.

Here we have used that, by the symmetry of the L\'evy measure $\mu(dz)$,  for any $ x\in \R^d,  r\in [0,t]$,
$$
\begin{aligned}
\mathrm{P.V.}\int_{|z|\leq 1}&\Big[f(r, x+A(x)z)-f(r, x)\Big]\,\mu(dz)
\\&=\int_{|z|\leq 1}\Big[f(r, x+A(x)z)-f(r, x)-{(D_2f)(r, X_r)A(x)z}\Big]\,\mu(dz).
\end{aligned}
$$
Next, we observe that the formula \eqref{Ito} remains true for any
$f=\left(f^1,\dots,f^d\right)$ such that $f^k \in C_b^{1,(1,\eta_2)}([0,t]\times \R^d),\ k=1,\dots,d$. Indeed, we can approximate $f$ in a standard way, e.g. by taking  convolutions in spatial variable Gaussian kernels, by a sequence of functions $f_n=\left(f_n^1,\dots,f_n^d\right)$ with $f^k_n\in C^{1,2}([0,t]\times \R^d), \ k=1,\dots,d, n\geq 1$ such that
$$
\sup_{n\geq 1}\rho^{[0,t]}_{\eta_2,1}(f^k_n)<\infty, \ k=1,\dots,d
$$
and $f^k, D_1f^k, D_2f^k$ are approximated uniformly on compact sets by $f^k_n, D_1f^k_n, D_2f^k_n,$ respectively. The formula \eqref{Ito} holds true for every $f_n, n\geq 1$, and both sides of this formula converge in probability as $n\to \infty$ to the corresponding terms in \eqref{Ito} for $f$. For the LHS and the first (deterministic) integral in the RHS this follows easily from the uniform convergence of $f^k_n, D_1f^k_n, D_2f^k_n$ on compacts. For the second (stochastic) integral this follows by the usual It\^o isometry arguments, convergence $f_n\to f$ and uniform bound on the increments $|f(s, x)-f(s,y)|\leq C|x-y|$, valid on every compact set. The third integral in the RHS is just an a.s. finite sum over the set of large jumps for the process $Z$, and thus the required converges follows simply from  $f_n\to f$. Finally, for the fourth integral the required convergence follows by the dominated convergence theorem because, by \eqref{eta_bound}, for $n\geq 1$, we have
$$
\Big|f_n(r, x+A(x)z)-f_n(r, x)-{(D_2f)(r, X_r)A(x)z}\Big|\leq C|z|^{1+\eta_2}, \quad  r\in [0,t],x\in \R^d, |z|\leq 1,
$$
and, by (\ref{bound_nu}),
$$
\int_{|z|\leq 1}|z|^{1+\eta_2}\mu(dz)<\infty
$$
since $1+\eta_2>\beta$.
%; see ??\footnote{We have to refer somewhere for the fact that $|z|^{\gamma}1_{|z|\leq 1}$ is integrable whenever $\gamma>\beta$}.

To conclude the proof, we simply take $f(s, x)=\chi_{t-s}(x)$. By Lemma \ref{chi_derivative} and Remark \ref{kappa_chi} below, for any $t>0$ this is a function with all coordinates   of the class  $C_b^{1,(1,\eta_2)}([0,t]\times \R^d)$ which satisfies
$$
D_1f(s, x)=-\partial_r\chi_{r}(x)\Big|_{r=t-s}=-D\chi_{r}(x)b(x)\Big|_{r=t-s}=-D_2f(s, x)b(x).
$$
Thus the formula \eqref{Ito} holds true for this function, and the first (deterministic) integral at its RHS equals zero. The 2nd, 3rd, and 4th integrals at the RHS of \eqref{Ito} coincide with the 2nd, 4th and 3rd terms at the RHS of \eqref{main1_bis} respectively, integrated in $s$.
\end{proof}

\begin{corollary}\label{c213}
The SDE (\ref{main}) has a unique weak solution $X$. This solution is a Markov process with the transition probability density $u_{t}(x,y)$ which satisfies \eqref{heat}.
\end{corollary}
\begin{proof}
  By Lemma \ref{c212}, for any weak solution  $X_s, s\in [0,t]$ to \eqref{main} with $X_0=x$, the process
  $$
  X_s^*=\chi_{-s}(X_{s+t})=\kappa_s(X_{s+t}), \quad s\in [-t,0]
  $$
  is a weak solution to SDE \eqref{main2} with $X_{-t}^*=\chi_t(x)$. For the latter equation the weak uniqueness follows by \cite[Theorem 2.1]{KKR2021}; we show in Lemma \ref{assumptions} below that all the assumptions of this theorem are satisfied, see also Remark \ref{rem_assumptions}. This gives weak uniqueness for \eqref{main}.

  Existence of a weak solution to \eqref{main} can be either derived from the similar result for \eqref{main2}, or proved directly by standard weak compactness arguments; we omit the details.

  Finally, we have
  $$
   X_0^*=\kappa_0(X_{t})=X_{t} \Longrightarrow \mathrm{Law}\, (X_t|X_{0}=x)=\mathrm{Law}\, (X_0^*|X_{-t}^*=\chi_t(x)),
  $$
  which proves \eqref{heat}.
\end{proof}

\subsection{Properties of the flow.} \label{flow}

In this Section we prove some properties $\kappa_s$ and $\chi_s$ needed in the sequel. Such properties seem to be standard in the theory of ODE, but we present them here for the convenience of the reader.

\begin{lemma}
\label{chi_derivative}
If asumptions \textbf{(C)} are satisfied
 then for any $t, s \in \R$, $x, y \in \R^d$,
 \begin{eqnarray}
\label{kappa_tsx}
|\kappa_t(x) - \kappa_s(x)| &\le& c |t - s| e^{c (|t| \vee |s|)} |x|,\\
\label{k_estimate}
|\kappa_t(x) -\kappa_t(y)| &\le& c |x-y|e^{c|t|},\\
\label{D_kappa_1}
|D\kappa_t(x) | &\le& ce^{c|t|},\\
\label{D_estimate}
|D\kappa_t(x) -D\kappa_t(y)| &\le& c |x-y|^{\eta_2}e^{c|t|}
\end{eqnarray}
and
%\mr {In the column notation:}
%For any $t > 0$, $x \in \R^d$ we have
%$$
%D\chi_t(x) b(x) = \partial_t \chi_t(x).
%$$
%D\kappa_t(x) b(x) = -\partial_t \kappa_t(x).
%$$

%\mr {In the row notation:}
 \begin{eqnarray}
\label{derivative_representation}
%\partial_t \kappa_t(x)=-b(x) [D\kappa_t(x)]^T.
\partial_t \kappa_t(x)&=&-D\kappa_t(x)b(x).%\\
%\label{exponent_representation}
%D\kappa_t(x)&=& \exp \int_0^t (-Db)(\kappa_s(x)) \, ds.
\end{eqnarray}
\end{lemma}
\begin{remark}
\label{kappa_chi}
 The lemma is true if we replace $\kappa_t(x)$ by $\chi_t(x)$ and $b(x)$ by $-b(x)$.

\end{remark}

\begin{proof} It is clear that it is enough to prove it for $t, s \ge 0$.
Note that
\begin{equation}
\label{kappa_t}
\kappa_{t}(x)= x - \int_0^t\ b(\kappa_{s}(x))ds.
\end{equation}
We observe that from the general theory $D\kappa_t(x)$ exists for any $x\in \R^d$. We have
\begin{eqnarray*}
|\kappa_{t}(x)-\kappa_{t}(y)|&=& \left|x -y -\int_0^t\ (b(\kappa_{s}(x))- b(\kappa_{s}(y)))ds\right|\\
&\le & |x -y| +\int_0^t\ |b(\kappa_{s}(x))- b(\kappa_{s}(y))|ds\\
&\le & |x -y| +C_7\int_0^t\ |\kappa_{s}(x)- \kappa_{s}(y)|ds.
\end{eqnarray*}
 %The statement (\ref{D_estimate}) is a standard property for the solutions of ODE which define $\kappa$
 Next we use Gronwall Lemma to get  (\ref{k_estimate}). As a consequence of this and the existence of $D\kappa_t(x)$ we obtain
\begin{equation}
\label{D_kappa}
|D\kappa_t(x) | \le ce^{ct}.
\end{equation}

Applying $D$ to the equality (\ref{kappa_t}),
then using  (\ref{D_kappa}) and (\ref{b_grad_est}) we have

\begin{equation}
\label{D_kappa1}
D\kappa_{t}(x)= I  -\int_0^t\  D b(\kappa_{s}(x))D \kappa_{s}(x)ds,
\end{equation}
where $I$ is identity matrix. This representation, the estimate (\ref{k_estimate}) and the arguments based on Gronwall Lemma leads to (\ref{D_estimate}).

Using the continuity of the function $s\mapsto  D b(\kappa_{s}(x))D \kappa_{s}(x)$ we have the following matrix differential equation for $y_t(x) = D\kappa_{t}(x)$:
\begin{equation}
\label{diff}
\frac{\partial}{\partial t} D\kappa_{t}(x)= -  D b(\kappa_{t}(x))D \kappa_{t}(x), \quad  y_0(x) = D\kappa_{0}(x)=I,
\end{equation}
 which can be written, denoting $J_t(x) = -D b(\kappa_{t}(x))$, as
$$\frac{\partial}{\partial t} y_t(x)= J_t(x)  y_t(x)  ,\quad  y_0(x) =I. $$
Now, apply to both sides  the vector  $b(x)$ to obtain
\begin{equation*}
\frac{\partial}{\partial t}  y_t(x)b(x)= J_t(x) y_t(x)b(x) ,\quad  y_0(x)b(x)= b(x).
\end{equation*}
On the other hand
\begin{eqnarray*}
\frac{\partial}{\partial t} b(\kappa_{t}(x)) =  D b(\kappa_{t}(x)) \frac{\partial}{\partial t}
\kappa_{t}(x)=
%-D b(\kappa_{t}(x)) b(\kappa_{t}(x))=
 J_t(x) b(\kappa_{t}(x)), \quad
  b(\kappa_{0}(x))=b(x).
\end{eqnarray*}
Finally, we observe  that  $z_1(t)=b(\kappa_{t}(x))$ and $z_2(t)=y_t(x)b(x) $ satisfy the same differentia equations with the same initial condition, hence they are equal. Therefore we have
$$\partial_t \kappa_{t}(x)= -b(\kappa_{t}(x))=-y_t(x) b(x)=  -D \kappa_{t}(x)b(x), $$
which ends the proof \eqref{derivative_representation}.

By (\ref{b_grad_est}) for any $x \in \R^d$ we have $|b(x)| \le c |x|$. Using this, (\ref{D_kappa_1}) and \eqref{derivative_representation} we get (\ref{kappa_tsx}).
\end{proof}

\begin{lemma}
If asumptions \textbf{(C)} are satisfied then there is $c$ such that for any $t\in \R$ and $x, y \in \R^d$ we have
\begin{equation}\label{A_estimate}
|A_t(x) - A_t(y)|\le c e^{c|t|}|x-y|^{\eta_1\wedge \eta_2}.
\end{equation}
\end{lemma}
\begin{proof} We recall that
$$A_t(x) = D \kappa_{t}(\chi_{t}(x)) A(\chi_{t}(x)).$$
Applying (\ref{D_estimate}) and (\ref{k_estimate}) we get
$$|D \kappa_{t}(\chi_{t}(x))- D \kappa_{t}(\chi_{t}(y))|\le c e^{c|t|}|x-y|^{ \eta_2}.$$
By (\ref{Holder}) and (\ref{k_estimate}),
$$|A(\chi_{t}(x))- A(\chi_{t}(y))|\le c e^{c|t|}|x-y|^{ \eta_1}.$$
Since
 $$\sup_{x\in \R^d}\left(|D \kappa_{t}(x)| + |A(x)|\right)\le c e^{c|t|}$$
we get (\ref{A_estimate}).
\end{proof}

\begin{lemma}
\label{CD}
If conditions \textbf{(C)} and \textbf{(D)} hold then condition \textbf{(D')} is satisfied, with a constant $C_9$ depending only on $d$, $\eta_1$, $\eta_2$, $C_1, \ldots, C_8$.
\end{lemma}
\begin{proof}
By (\ref{kappa_t}) we obtain
$$
\kappa_{t}(x)- \kappa_{s}(x)=- \int_s^t b(\kappa_{u}(x))  du
$$
and then
$$
D\kappa_{t}(x)- D\kappa_{s}(x)=- \int_s^t   Db(\kappa_{u}(x))D\kappa_u(x) \, du.
$$
Since
$$
| D\kappa_u(x) |\le  c e^{c|u|}
$$
we arrive at
$$
|\kappa_{t}(x)- \kappa_{s}(x)|\le || b||_\infty|t-s| $$
and
$$|D \kappa_{t}(x)- D \kappa_{s}(x)|\le c e^{c(|s|\vee |t|)}|t-s|.$$
These yield
\begin{eqnarray*}&& |D \kappa_{t}(\chi_{t}(x))- D \kappa_{s}(\chi_{s}(x))|\\
&\le& |D \kappa_{t}(\chi_{t}(x))- D \kappa_{t}(\chi_{s}(x))|+ |D \kappa_{t}(\chi_{s}(x))- D \kappa_{s}(\chi_{s}(x))|\\
&\le& c e^{c(|s|\vee |t|)} |\kappa_{-t}(x)- \kappa_{-s}(x)|^{\eta_2} +  c e^{c(|s|\vee |t|)}|t-s|\\
&\le& c e^{c(|s|\vee |t|)}[||b||^{\eta_2}_\infty|t-s|^{\eta_2} + |t-s|].
 \end{eqnarray*}
Moreover,
$$|A(\chi_{t}(x))- A(\chi_{s}(x))|\le C_5|\chi_{t}(x)-\chi_{s}(x)|^{ \eta_1}\le  C_5||b||^{\eta_1}_\infty|t-s|^{ \eta_1}.$$
Hence we obtain (\ref{D11}).
\end{proof}

%\begin{remark} Let $b(x)=-x$. Then $\chi_t(x)= xe^{-t}$, $\kappa_t(x)= xe^{t}$. In this case
%$(D \kappa_{t})(x)= Ie^{t}$,  hence
 %$$A_t(x) =  e^{t} A(xe^{t}).$$
%This example shows that the above estimates can not be improved. Moreover, we can not rely on the results of \cite{KKR2021} directly. Some modification is required.
%\end{remark}

Let $a_{t,i,j}(x)$ be elements of $A_t(x)$ (where $t \in \R$, $i, j \in \{1,\ldots,d\}$).

\begin{lemma}
\label{assumptions}
Put $\gamma_1 = \gamma_2 = \min(\eta_1,\eta_2)$, $\gamma_3 = 1+\eta_2$.
If asumption \textbf{(C)} is satisfied then for any $x, y \in \R^d$, $t, s \in \R$ we have
\begin{eqnarray}
\label{atijx}
|a_{t,i,j}(x)| &\le& c e^{c |t|},\\
\label{det_At}
|\det(A_t(x))| &\ge& c e^{-c |t|},\\
\label{atijxy}
|a_{t,i,j}(x) - a_{t,i,j}(y)| &\le& c e^{c |t|}|x-y|^{\gamma_1}.
\end{eqnarray}
We also have
\begin{eqnarray}
\label{gamma3}
\gamma_3 &>& \max(1,\beta),\\
\label{V_t_A_t}
  |V_t(x,z)-A_t(x)z|&\le& c e^{c|t|}|z|^{\gamma_3}.
\end{eqnarray}
If (\ref{indices}) is satisfied then
\begin{equation}
\label{gamma123}
\frac{\beta}{\alpha} < 1 + \gamma_1, \quad \frac{1}{\alpha} - \frac{1}{\beta} < \gamma_2, \quad \frac{\beta}{\alpha} < \gamma_3.
\end{equation}
\end{lemma}
\begin{remark}
\label{rem_assumptions}
Note that \textbf{(D')} clearly implies that for any $x, y \in \R^d$, $t, s \in \R$ we have
$$
|a_{t,i,j}(x) - a_{s,i,j}(x)| \le c e^{c \max(|t|,|s|)}|t-s|^{\gamma_2}.
$$
This and the above lemma shows that assumptions on $A_t(x)$ and $V_t(x,z)$ stated in (8-15) in \cite{KKR2021} are satisfied (locally in $t$).
\end{remark}
\begin{proof}[Proof of Lemma \ref{assumptions}]
By (\ref{bounded}) and (\ref{D_kappa_1}) we obtain (\ref{atijx}).

Put $J_t(x) = -Db(\kappa_t(x))$. For any $t \in \R$, $x \in \R^d$ by (\ref{b_grad_est}) we obtain $|\text{tr} \, J_t(x)| \le C_6 d$, which implies $\left|\int_0^t \text{tr} \, J_s(x) \, ds\right| \le C_6 d |t|$. By (\ref{diff}) and \cite[Theorem 7.3]{CL1955} we obtain
$$
|\det(D\kappa_t(x))|
= \exp\left(\int_0^t \text{tr} \, J_s(x) \, ds\right)
\ge e^{-C_6 d |t|}.
$$
Using this and (\ref{determinant}) we get (\ref{det_At}).

The estimate (\ref{atijxy}) follows from \eqref{A_estimate}, while  (\ref{eta2}) implies (\ref{gamma3}).

Recall that
$$V_t(x,z) = \kappa_{t}(\chi_{t}(x) + A(\chi_{t}(x)) z) -x= \kappa_{t}(\chi_{t}(x) + A(\chi_{t}(x)) z)-\kappa_{t}(\chi_{t}(x)).$$
We apply Remark \ref{Eta_bound}  to the coordinates of the vector functions $\kappa_{t}$ with $y$ replaced by  $\chi_{t}(x) + A(\chi_{t}(x)) z$ and $x$ replaced by $\chi_{t}(x)$, hence we get, by (\ref{b_estimate}),
$$
|V_t(x,z)-A_t(x)z|\le c e^{c|t|} |A(\chi_{t}(x))z|^{1+\eta_2}.
$$
Since $|A(\chi_{t}(x))z|\le C_3 |z|$ we get (\ref{V_t_A_t}).

Finally, (\ref{gamma123}) follows from elementary calculations.
\end{proof}

\subsection{Proofs of the main results}
In the whole section we assume either \textbf{(A)}, \textbf{(C)}, \textbf{(D')} or \textbf{(B)}, \textbf{(C)}, \textbf{(D')}, \textbf{(E)}.

\begin{proof}[Proof of Theorem \ref{main_thm} and Remark \ref{rem_main_thm}]
By Corollary \ref{c213}, for  any $t > 0$, $x, y \in \R^d$ we have
\begin{equation}
\label{main_equality}
u_{t}(x,y) = p_{-t,0}(\chi_{t}(x),y).
\end{equation}
Note that for any $x \in \R^d$ and $t > 0$ we have
$$
\int_{\R^d} u_t(x,y) \, dy = \int_{\R^d} \wt u_t(x,y) \, dy = \int_{\R^d} \check{u}_t(x,y) \, dy = 1,
$$
so it is enough to show (\ref{q_estimate}) for $t \in (0,1]$.
We need to recall some notation from \cite{KKR2021}. Let $G_t$ be the density defined in Section 5 in \cite{KKR2021} (roughly speaking, it is a modifed version of the density $\wt G_t$). For $t, s \in \R$, $s > t$, $x, w \in \R^d$ put (cf. \cite[(5.3)]{KKR2021})
$$
p_{t,s}^x(w) = \frac{1}{|\det(A_s(x))|} G_{s - t}(A_s^{-1}(x) w).
$$
We also define,  for $t, s \in \R$, $s > t$, $x, w \in \R^d$,
$$
\overline{p}_{t,s}^x(w) = \frac{1}{|\det(A_s(x))|} \wt G_{s - t}(A_s^{-1}(x) w).
$$

To the end of the proof we assume that $ |s| \le 1$,  $0 < s - t \le 1$ and $x \in \R^d$.
By \cite[(3.13) and (5.4)]{KKR2021},  we have
\begin{equation}
\label{p_estimate11}
\int_{\R^d} |p_{t,s}(x,y) - p_{t,s}^y(x-y)| \, dy \le c (s - t)^{\eps_0}.
\end{equation}
By Lemma 5.8 from \cite{KKR2021}  we obtain
\begin{equation}
\label{p_estimate12}
\int_{\R^d} |p_{t,s}^y(x-y) - p_{t,s}^x(x-y)| \, dy \le c (s - t)^{\eps_0},
\end{equation}
while Lemma 4.8 from \cite{KKR2021} yields
%$$
%\int_{\R^d} |p_{t,s}^x(x-y) - \overline{p}_{t,s}(x,y)| \, dy \le c (s - t)^{\eps_0}.
%$$}}
\begin{equation}
\label{p_estimate13}
\int_{\R^d} |p_{t,s}^x(w) - \overline{p}_{t,s}^x(w)| \, dw \le c (s - t)^{\eps_0}.\end{equation}
Note that
\begin{equation}
\label{tilde_q}
\overline{p}^{\chi_t(x)}_{-t,0}(\chi_t(x)-y) = \frac{1}{|\det(A(\chi_t(x)))|} \wt G_t(A^{-1}(\chi_t(x)) (y - \chi_t(x))) = \wt u_t(x,y)
\end{equation}
for any $t > 0$, $x, y \in \R^d$. Using this, (\ref{main_equality}) and (\ref{p_estimate11}, \ref{p_estimate12},
\ref{p_estimate13})  we get (\ref{q_estimate}).

Now assume that  \textbf{(D)} holds. Then one can easily prove
\begin{equation}
\label{p_estimate14}
\int_{\R^d} |p_{t,s}^x(w) - {p}_{t,s}^{\chi_{t-s}(x)}(w)| \, dw \le c (s - t)^{\eps_0}.\end{equation}
The above estimate can be obtained by slight modifications in the proof of Lemma 5.9 from \cite{KKR2021}.  Here we also need the bound $|\chi_{t-s}(x)-x|\le c |t-s|$, which follows from (\ref{kappa_t}) and is a consequence of the assumption \textbf{(D)}.

Combining the estimates (\ref{p_estimate11}, \ref{p_estimate12},
\ref{p_estimate13}, \ref{p_estimate14})  we obtain
\begin{equation}
\label{p_estimate1}
\int_{\R^d} |p_{t,s}(x,y) -  \overline{p}_{t,s}^{\chi_{t-s}(x)}(x-y)| \, dy \le c (s - t)^{\eps_0}.
\end{equation}
Note that
$$
\overline{p}_{-t,0}^{\chi_{-t}(\chi_{t}(x))}(\chi_t(x)-y) =
\overline{p}_{-t,0}^{x}(\chi_t(x)-y) =
\frac{1}{|\det(A(x))|} \wt G_t(A^{-1}(x) (y - \chi_t(x))).
$$
 Using this, (\ref{main_equality}) and (\ref{p_estimate1}) (with $x$ replaced by $\chi_t(x)$) we get (\ref{q_estimate}) in this case with $\wt u_t(x,y)$ replaced by $\check{u}_t(x,y)$.

\end{proof}

\begin{proof}[Proof of Theorem \ref{main_thm2}]
Fix $\tau > 0$.

By (\ref{main_equality}) and Theorem \cite[Theorem 2.2]{KKR2021} we obtain
$$
\sup_{x,y \in \R^d} |u_t(x,y)| = \sup_{x,y \in \R^d} |p_{-t,0}(\chi_{t}(x),y)| < \infty, \quad 0 < t < \infty.
$$
For $t, s \in \R$, $0 < s - t \le \tau$, $x, y \in \R^d$ we have
\begin{eqnarray*}
|p_{t,s}(x,y) -  \overline{p}^x_{t,s}(x-y)| &\le&
|p_{t,s}(x,y) - p_{t,s}^y(x-y)| \\
&& + |p_{t,s}^y(x-y) - p_{t,s}^x(x-y)| + |p_{t,s}^x(x-y) - \overline{p}^x_{t,s}(x-y)|.
\end{eqnarray*}
Using \cite[Lemmas 4.8 and 5.8]{KKR2021} and some arguments from Section 3.5 from \cite{KKR2021} all these terms are bounded by $c (s - t)^{\eps_0}$, where $c$ depends only on $d$, $\alpha$, $\beta$, $\eta_1$, $\eta_2$, $C_0, \ldots, C_8$, $\sigma(\tau)$. Using this, (\ref{main_equality}) and (\ref{tilde_q}) we get (\ref{q_pointwise_estimate}).
\end{proof}

\begin{proof}[Proof of Theorem \ref{Holder_thm_new}]
Fix $\tau >0$, $\gamma \in (0,\alpha)$ such that $\gamma \le 1$ and $\gamma' \in (\gamma,\alpha)$.
 Let $f: \R^d \to \R$ be a bounded Borel function.
Then for $t > 0$, $x \in \R^d$ we have
$$
U_t f(x) = \int_{\R^d} u_t(x,z) f(z) \, dz = \int_{\R^d} p_{-t,0}(\chi_{t}(x),z) f(z) \, dz = P_{-t,0}f(\chi_{t}(x)).
$$
Using this and \cite[Theorem 2.3]{KKR2021} we obtain for  $x,y \in \R^d$, $t \in (0,\tau]$
$$
|U_t f(x) - U_t f(y)| = |P_{-t,0}f(\chi_{t}(x)) - P_{-t,0}f(\chi_{t}(y))|
\le c |\chi_{t}(x) - \chi_{t}(y)|^{\gamma} t^{-\gamma'/\alpha} \|f\|_{\infty},
$$
where $c$ depends on $\gamma$, $\gamma'$, $d$, $\alpha$, $\beta$, $\eta_1$, $\eta_2$, $C_1, \ldots, C_7$, $C_9$, $\sigma(\tau)$.

By Lemma \ref{chi_derivative} we get $|\chi_{t}(x) - \chi_{t}(y)| \le c e^{c \tau} |x-y|$, which implies the assertion of the theorem.
\end{proof}

\section{Examples}
\label{sec_Examples}
Let us give several examples illustrating various specific issues  of the model. Note that a simplest  example of a L\'evy measure $\nu\in \mathrm{WSC}(\alpha, \beta)$  is a symmetric $\alpha_0$-stable L\'evy measure
\begin{equation}\label{alpha}
\nu(dx)=c\frac{dx}{|x|^{\alpha_0+1}},
\end{equation}
for which
$$
h(r)=\frac{4c}{\alpha_0(2-\alpha_0)} r^{-\alpha_0}
$$
and thus \eqref{h-scaling_u0} holds true with $\beta=\alpha = \alpha_0$ and $C_1=C_2=1$.

\begin{example}
Let $Z_t = (Z_t^{1},\ldots,Z_t^{d})^T$ be such that $Z_t^{1},\ldots,Z_t^{d}$ are independent and for each $i \in \{1,\ldots,d\}$ $Z_t^{i}$ is a one-dimensional, symmetric $\alpha_0$-stable process, where $\alpha_0 \in (0,2)$. Put $\alpha = \beta = \alpha_0$. Then assumptions \textbf{(A)} are satisfied.

One can compare the results in such case with results from \cite{CHZ2020}. In general, their results do not imply ours and our results do not imply theirs. On one hand the results from \cite{CHZ2020} allow to consider matrices $A$ and drifts $b$ to be dependent on time and the drifts can be H{\"o}lder continuous with respect to $x$ variable. These assumptions are less restrictive than ours. However, it is assumed in \cite{CHZ2020} that matrices must be differentiable in $x$ variable and $\alpha \in (1/2,2)$ (when drift is nontrivial). These assumptions are more restrictive than ours. One has also to add that the assertions on the transition density and the corresponding semigroup in \cite{CHZ2020} are stronger than ours.
\end{example}

\begin{example}
Let $Z_t = (Z_t^{1},\ldots,Z_t^{d})^T$ be such that $Z_t^{1},\ldots,Z_t^{d}$ are independent and for each $i \in \{1,\ldots,d\}$ $Z_t^{i}$ is a one-dimensional, symmetric $\alpha_i$-stable process ($\alpha_i \in (0,2)$ and they are not all equal). Put $\alpha = \min(\alpha_1,\ldots,\alpha_d)$ and $\beta = \max(\alpha_1,\ldots,\alpha_d)$.  Then assumptions \textbf{(B)} are satisfied. The SDE (\ref{main}) driven by such process $Z$ and the corresponding anisotropic generator are of great interest see e.g. \cite{C2020}, \cite{C2019}, \cite{CK2020}, \cite[example (Z2) on page 2]{FJR2021}. Note that results from \cite{CHZ2020} do not allow to consider SDE driven by such process.
\end{example}

The next example shows that under just the basic assumptions of Theorem  \ref{main_thm}, the transition probability density may be locally unbounded.

\begin{example} (See \cite[Remark~4.23]{KRS2021}, \cite[Example~4.2]{KKS2021}).
  Let $Z_t = (Z_t^{1},\ldots,Z_t^{d})^T$ be such that $Z_t^{1},\ldots,Z_t^{d}$ are independent and for each $i \in \{1,\ldots,d\}$ $Z_t^{i}$ is a one-dimensional, symmetric $\alpha_0$-stable process, where $\alpha_0 \in (0,2)$, $b \equiv 0$ and
  matrices $A(x)$ are H\"older continuous in $x$, for each $x \in \R^d$ the matrix
  $A(x)$ is a rotation (hence, an isometry) and for any $x$ in some open cone with vertex at $0$, which satisfies $|x| \ge 1$ we have $A(x)\mathbf{e}_1 = x/|x|$. Then assumptions \textbf{(B)}, \textbf{(C)}, \textbf{(D)} are satisfied, so by Theorem \ref{main} the transition density $u_t(x,y)$ exists. However, for $\alpha+1\leq d$, for any $x\in \R^d$ the transition density $u_t(x,y)$ is unbounded at any neighbourhood of the point $y=0$.
\end{example}

Now, we present example of matrices $A$ and a drift $b$ for which assumption \textbf{(D)} is not satisfied (the drift is unbounded) but for which assumptions \textbf{(C)}  and \textbf{(D')} hold.

\begin{example} \label{Ex_D}
Assume that $b(x) = B x$ for some constant $d \times d$ matrix $B$ and any $x \in \R^d$. Assume $d \times d$ matrices $A(x)$ satisfy (\ref{bounded}), (\ref{determinant}), (\ref{Holder}) with $\eta_1 = 1$ and additionally $A(x)$ is constant when $|x|$ is large. More precisely, there exists a $d \times d$ matrix $J$ and $p_0 > 0$ such that for any $x \in \R^d$ if $|x| \ge p_0$ then $A(x) = J$.

Assume also that $\beta  \in (0,2)$. Then one can show that conditions \textbf{(C)} (with $\eta_2 = 1$) and \textbf{(D')} are satisfied.

Indeed, $\eta_2 = 1 > \min(0,\beta-1)$, $Db(x) = B$ so (\ref{eta2}), (\ref{b_grad_est}), (\ref{b_estimate}) are satisfied. This implies that conditions \textbf{(C)} holds. We also have $\kappa_t(x) = e^{-Bt} x$, $\chi_t(x) = e^{Bt} x$. One can easily show that there exists $p_1 > 0$ such that for any $t \in \R$, $x \in \R^d$ we have
$$
|e^{Bt} x| \ge \frac{1}{p_1} e^{-p_1|t|}|x|.
$$
We may assume that $|t - s| \le 1$ and $|s| < |t|$.

We have $A_t(x) = e^{-Bt} A(e^{Bt} x)$. Note that
$$
|A(e^{Bt}x) - A(e^{Bs}x)| \le c |(e^{Bt} - e^{Bs}) x| \le c |t - s| e^{c |t|} |x|.
$$
If $|x| \ge p_0 p_1 e^{p_1 |t|}$ then
$$
|e^{Bt}x| \ge \frac{1}{p_1} e^{-p_1|t|}|x| \ge p_0, \quad \text{and} \quad
|e^{Bs}x| \ge \frac{1}{p_1} e^{-p_1|t|}|x| \ge p_0,
$$
so $|A(e^{Bt}x) - A(e^{Bs}x)| = 0$. On the other hand, if $|x| \le p_0 p_1 e^{p_1 |t|}$ then
$$
|A(e^{Bt}x) - A(e^{Bs}x)| \le  c p_1 p_0 |t - s| e^{(c + p_1) |t|}.
$$
This clearly implies (\ref{D11}).
\end{example}

Our last example explains why in the case \textbf{(B)}, i.e. for a cylindrical noise which has \emph{different} scaling indices of the coordinates,  additional assumption \textbf{(E)} on the H\"older indices of the coefficients should be made. This example is similar to Example 2.7 from \cite{KKR2021}.

\begin{example}
\label{ex2}
  Let $Z^i,i=1, 2$ be independent, symmetric $\alpha_i$-stable processes with $0 < \alpha_1 < \alpha_2 \le 1$. The process $Z=(Z^1, Z^2)^T$ fits to our case  \textbf{(B)} with $\alpha=\alpha_1$, $\beta=\alpha_2$. In this example, we show that  in such -- extremely spatially non-homogeneous -- setting the additional assumption \textbf{(E)} is crucial in the sense that, without this condition, the structure of the transition density can be quite different.

  For any $x \in \R^2$ let $A(x)$ be the identity matrix and $b(x) = b\begin{pmatrix}  x_1\\   x_2 \end{pmatrix} = \begin{pmatrix}  x_2\\   -x_1 \end{pmatrix}$. Then for any $x \in \R^2$ and $t \in \R$ we have
	$$
	\kappa_t(x) = \begin{pmatrix}  x_1 \cos t + x_2 \sin t\\   -x_1 \sin t + x_2 \cos t \end{pmatrix}.
	$$
	It follows that
  $$
 A_t(x) =\left(
                                        \begin{array}{cc}
                                          \cos t & \sin t \\
                                          - \sin t & \cos t \\
                                        \end{array}
                                      \right)
  $$
Note that $A_t(x)$ does not depend on $x$ so we will write $A_t$ instead of $A_t(x)$. We have $\kappa_t(x) = A_t x$ and $\chi_t(x) = \kappa_{-t}(x) = A_{-t} x$.

Note that assumptions \textbf{(B)}, \textbf{(C)}, \textbf{(D')} are satisfied with $\eta_1 = \eta_2 = 1$. For any $x, z \in \R^d$ and $t \in \R$ we have
$$
V_t(x,z) = \kappa_{t}(\chi_{t}(x) + A(\chi_{t}(x)) z) - x = \kappa_t(A_{-t} x + z) - x = A_t z
$$
Therefore the additional assumption  \textbf{(I)} holds true since each  operator $T^{t,z}$ is just an isometry which corresponds to the shift of the variable $x\mapsto x+A_tz$. The assumption \textbf{(E)} is equivalent to
\begin{equation}
\label{1alpha1beta}
\frac{1}{\alpha} - \frac{1}{\beta} < 1.
\end{equation}
If this condition is satisfied then assertions of Theorems \ref{main_thm} and \ref{main_thm2} hold. Then using (\ref{Euler}) and these theorems one has for any $t \in (0,1]$ and $x \in \R^d$
$$
\wt{u}_{t}(0,0) = \wt{G}_{t}(0)  = c t^{-\frac{1}{\alpha}-\frac{1}{\beta}},
$$
$$
c t^{-\frac{1}{\alpha}-\frac{1}{\beta}} - c_1 t^{-\frac{1}{\alpha}-\frac{1}{\beta} + \eps_0} \le u_t(0,0) \le c t^{-\frac{1}{\alpha}-\frac{1}{\beta}} + c_1 t^{-\frac{1}{\alpha}-\frac{1}{\beta} + \eps_0},
$$
which implies that there exists sufficiently small $\tau \in (0,1]$ and $c_2 > 1$ such that for any $t \in (0,\tau]$ we have
\begin{equation}
\label{bounds_near_0}
c_2^{-1} \le \frac{u_t(0,0)}{t^{-\frac{1}{\alpha}-\frac{1}{\beta}}} \le c_2.
\end{equation}
For $\frac{1}{\alpha} - \frac{1}{\beta} > 1$ the situation changes drastically; namely, we have for $t \in (0,\pi/6]$
\begin{equation}
\label{on-diagonal-est}
\frac{u_{t}(0,0)}{t^{-\frac{1}{\alpha}-\frac{1}{\beta}}} \le c t^{\frac{1}{\alpha} - \frac{1}{\beta} -1} \to 0 \quad \text{as $t \to 0^+$,}
\end{equation}
so (\ref{bounds_near_0}) does not hold.

Now we will prove (\ref{on-diagonal-est}). The justification is similar to the proof of (19) in Appendix A in \cite{KKR2021}. We have $u_t(0,0) = p_{-t,0}(\kappa_{-t}(0),0) = p_{-t,0}(0,0)$. Recall that for $-\infty < t < s < \infty$ and $x \in \R^d$ $p_{t,s}(x,\cdot)$ is the transition density of the solution of the SDE
$$
dX_r =  A_r dZ_r, \quad \text{$X_t = x, \, r \ge t$.}
$$
Hence, for $t > 0$ $p_{-t,0}(0,\cdot)$ is the transition density of the solution of the SDE
\begin{equation}
\label{SDE_Ar}
dX_r =  A_r dZ_r, \quad \text{$X_{-t} = 0, \, r \ge -t$.}
\end{equation}
Let $X_r$ be the solution of (\ref{SDE_Ar}) and $Z_r = (Z_r^1,Z_r^2)$. We have
$$
X_0 = \int_{-t}^0 A_r dZ_r
= \begin{pmatrix}
\int_{-t}^0 \cos r dZ_r^1 + \int_{-t}^0 \sin r dZ_r^2 \\
-\int_{-t}^0 \sin r dZ_r^1 + \int_{-t}^0 \cos r dZ_r^2 \end{pmatrix}.
$$
The characteristic function of $X_0$ has the form
$$
\phi(z)=\exp\left\{-\int_{-t}^0(|(A_r z)_1|^{\alpha_1}+|(A_r z)_2|^{\alpha_2})\, dr\right\},
$$
where $(A_r z)_1 = z_1 \cos r - z_2 \sin r$, $(A_r z)_2 = z_1 \sin r + z_2 \cos r$.
Thus the distribution density equals
$$
p(x)=\frac{1}{(2\pi)^2}\int_{\R^2}\exp\left\{-ix\cdot z-\int_{-t}^0(|(A(r)z)_1|^{\alpha_1}+|(A(r)z)_2|^{\alpha_2})\, dr\right\}\, dz.
$$
In particular,
$$
p(0)=\frac{1}{(2\pi)^2}\int_{\R^2}\exp\left\{-\int_{-t}^0(|(A(r)z)_1|^{\alpha_1}+|(A(r)z)_2|^{\alpha_2})\, dr\right\}\, dz.
$$

For $t \in [0,\pi/6]$ let us  estimate from below
$$
\int_{-t}^0(|(A_rz)_1|^{\alpha_1}+|(A_rz)_2|^{\alpha_2})\, dr\geq { \int_{-t}^0|(A_rz)_2|^{\alpha_2}\, dr}.
$$
We will consider 2 cases.

\emph{Case 1:} $|z_2|>\frac{t}2|z_1|$.

Then for $r \in [-t/8,0]$ we have
$$
|(A_rz)_2| = |z_1 \sin r + z_2 \cos r| \ge |z_2 \cos r| - |z_1 \sin r|.
$$
We have $|z_1 \sin r| \le |r| |z_1| \le t |z_1|/8 \le |z_2|/4$ and $|z_2 \cos r| \ge \cos(\pi/6) |z_2| \ge |z_2|/2$.
This  gives
$$
\int_{-t}^0|(A_rz)_2|^{\alpha_2}\, dr \geq c t |z_2|^{\alpha_2}.
$$

\emph{Case 2:} $|z_2|\leq \frac{t}2|z_1|$.

For $r \in [-t,-3t/4]$ we have
$$
|(A_rz)_2| = |z_1 \sin r + z_2 \cos r| \ge |z_1 \sin r| - |z_2 \cos r| \ge \sin\left(\frac{3t}{4}\right) |z_1| - \frac{t |z_1|}{2}
\ge \left(\frac{9}{4 \pi} - \frac{1}{2}\right) t |z_1|.
$$
which gives
$$
{\int_{-t}^0|(A_rz)_2|^{\alpha_2}\, dr\geq c t^{1+\alpha_2} |z_1|^{\alpha_2}}.
$$

Now we can complete the estimate of $p(0)$. We have
$$
\begin{aligned}
p(0)&\leq \frac{1}{(2\pi)^2}\int_{\R^2}\exp\left\{-\int_{-t}^0|(A_r z)_2|^{\alpha_2}\, dr\right\}\, dz
\\&\leq \frac{1}{(2\pi)^2} \int_{|z_2|> \frac{t}2|z_1|}\exp\left\{-c t |z_2|^{\alpha_2}\right\}\, dz
+ \frac{1}{(2\pi)^2}\int_{|z_2|\leq \frac{t}2|z_1|}\exp\left\{-c t^{1+\alpha_2} |z_1|^{\alpha_2}\right\}\, dz=:I_1+I_2.
\end{aligned}
$$
Since
$$
I_1=\frac{4}{(2\pi)^2 t}\int_{\R}|z_2|\exp\left\{-c t |z_2|^{\alpha_2}\right\}\, dz_2= c t^{-1-{2/\alpha_2}},
$$
$$
I_2=\frac{t}{(2\pi)^2}\int_{\R}|z_1|\exp\left\{-c t^{1+\alpha_2} |z_1|^{\alpha_2}\right\}\, dz_1= c t^{-1-{2/\alpha_2}},
$$
this completes the proof of \eqref{on-diagonal-est}.
\end{example}

%%%%%%%%%%%%%%%%%%%%%%%%%%%%%%%%%%%%%%%%%%%%%%
%% Single Appendix:                         %%
%%%%%%%%%%%%%%%%%%%%%%%%%%%%%%%%%%%%%%%%%%%%%%
%\begin{appendix}
%\section*{???}%% if no title is needed, leave empty \section*{}.
%\end{appendix}
%%%%%%%%%%%%%%%%%%%%%%%%%%%%%%%%%%%%%%%%%%%%%%
%% Multiple Appendixes:                     %%
%%%%%%%%%%%%%%%%%%%%%%%%%%%%%%%%%%%%%%%%%%%%%%
%\begin{appendix}
%\section{???}
%
%\section{???}
%
%\end{appendix}

%%%%%%%%%%%%%%%%%%%%%%%%%%%%%%%%%%%%%%%%%%%%%%
%% Support information, if any,             %%
%% should be provided in the                %%
%% Acknowledgements section.                %%
%%%%%%%%%%%%%%%%%%%%%%%%%%%%%%%%%%%%%%%%%%%%%%
%\begin{acks}[Acknowledgments]
% The authors would like to thank ...
%\end{acks}
%%%%%%%%%%%%%%%%%%%%%%%%%%%%%%%%%%%%%%%%%%%%%%
%% Funding information, if any,             %%
%% should be provided in the                %%
%% funding section.                         %%
%%%%%%%%%%%%%%%%%%%%%%%%%%%%%%%%%%%%%%%%%%%%%%
\textbf{{Acknowledgements}}
The first and the third author were supported in part by the National Science Centre, Poland, grant no. 2019/35/B/ST1/01633

The second author has been supported through
the DFG-NCN Beethoven Classic 3 programme, contract no.
2018/31/G/ST1/02252 (National Science Center, Poland) and SCHI-419/11-1
(DFG, Germany)

\end{document}